\documentclass[fullpage]{article}
\usepackage{fullpage}
\usepackage[utf8]{inputenc}
\usepackage[T2A]{fontenc}
\usepackage{subfigure}
\usepackage{amsmath}

\usepackage{subfigure}

\usepackage[
    backend=biber,
    style=alphabetic,
    useprefix=true, 
    bibencoding=utf8,
    maxbibnames=10 
]{biblatex}
\bibliography{./fpsac}

\usepackage{enumitem}
  \setlist{nosep} 
\usepackage{xspace}
\usepackage{graphicx}
\usepackage{bm}
\usepackage{comment}
\usepackage{float}

\usepackage{url}
\usepackage{breakurl}
\usepackage[breaklinks]{hyperref}
\usepackage[
    capitalize,
    nameinlink,
    noabbrev
]{cleveref}

\usepackage{xcolor}
\definecolor{lgreen}{rgb}{0.0, 0.48, 0.0}
\definecolor{lpurple}{rgb}{0.48, 0.0, 0.48}
\hypersetup{linktocpage,
            colorlinks=true,
            linkcolor=lgreen,
            citecolor=lpurple}

\usepackage{amsthm}
\theoremstyle{definition}

\newtheorem{theorem}{Theorem}
\newtheorem{proposition}[theorem]{Proposition}
\newtheorem{lemma}[theorem]{Lemma}

\newtheorem{remark}[theorem]{Remark}

\usepackage{amssymb}
\renewcommand{\geq}{\geqslant}

\newcommand{\bigO}{\mathcal{O}}

\newcommand*{\ie}{\textit{i.e.}\@\xspace}

\newcommand{\exponential}[1]{\operatorname{#1}}

\newcommand{\eG}{\exponential{G}}
\newcommand{\eS}{\exponential{S}}

\newcommand{\eComplex}{\exponential{Complex}}

\newcommand{\eF}{\exponential{F}}

\newcommand{\graphic}[1]{\mathbf{#1}}

\newcommand{\gF}{\graphic{F}}

\newcommand{\gset}{\graphic{Set}}

\newcommand{\gdag}{\graphic{DAG}}

\newcommand{\hadamard}{\odot}

\newcommand{\family}[1]{\mathcal{#1}}

\newcommand{\fF}{\family{F}}

\newcommand{\fdag}{\family{DAG}}

\newcommand{\fsmd}{\family{SMD}}

\definecolor{bblue}{rgb}{0.2, 0.4, 0.8}
\definecolor{bgreen}{rgb}{0.2, 0.6, 0.4}
\definecolor{bred}{rgb}{0.8, 0.4, 0.2}
\definecolor{bviolet}{rgb}{0.7, 0.2, 0.7}
\definecolor{blackred}{rgb}{0.6, 0.3, 0.3}
\definecolor{blackblue}{rgb}{0.3, 0.3, 0.6}

\usepackage{tikz}
\usetikzlibrary{fit,arrows,trees,shapes,shapes.geometric,calc,matrix,
decorations.markings}
\tikzset{
  treenode/.style = {align=center, inner sep=0pt, text centered,
    font=\sffamily},
  arnBleuPetit/.style = {treenode, circle, bblue, draw=bblue,
    fill=bblue!10,
    minimum width=0.8em, minimum height=0.5em
  },
  arnNoirPetit/.style = {treenode, triangle, draw=black,
    minimum width=0.8em, minimum height=0.5em},
  arnBleuGrande/.style = {treenode, circle, bblue, draw=bblue,
    text width=1.5em, very thick,
    fill=bblue!10},
  arnVioletGrande/.style = {treenode, circle, bviolet, draw=bviolet,
    text width=1.5em, very thick,
    fill=bblue!10},
}

\title{Counting directed acyclic and elementary digraphs}

\date{\today}

\author{\'Elie de Panafieu\thanks{\'Elie de Panafieu is supported by LINCS www.lincs.fr},
Bell Labs France, Nokia\\
Sergey Dovgal\thanks{dovgal@lipn.fr,
Sergey Dovgal is supported by ANR project METACONC.},
LIPN, Universit\'{e} Paris 13}

\begin{document}

\maketitle

\begin{abstract}
Directed acyclic graphs (DAGs) can be characterised as directed graphs whose
strongly connected components are isolated vertices. Using this restriction on
the strongly connected components, we discover that when $m = cn$, where $m$ is the number of
directed edges, $n$ is the number of vertices, and $c < 1$, the asymptotic
probability that a random digraph is acyclic is an explicit
function $p(c) = e^{-c}(1-c)$.
When $m = n(1 + \mu n^{-1/3})$, the
asymptotic behaviour changes, and the probability that a digraph is acyclic
becomes $n^{-1/3} C(\mu)$, where $C(\mu)$ is an explicit function of $\mu$.
\L uczak and Seierstad (2009, Random Structures \& Algorithms, 35(3), 271--293)
showed that, as $\mu \to -\infty$, the strongly
connected components of a random digraph with $n$ vertices and $m = n(1 + \mu n^{-1/3})$ directed edges
are, with high probability, only isolated vertices and cycles. We call such
digraphs elementary digraphs. We express the probability that a random digraph
is elementary as a function of $\mu$.
Those results are obtained using techniques from analytic combinatorics,
developed in particular to study random graphs. 
\end{abstract}

\begin{section}{Introduction}
\label{section:introduction}
    Directed Acyclic Graphs (DAGs) appear
    naturally in the study of compacted trees, automaton for 
    finite languages and partial orders.
    Until now, their asymptotics was
    known only 
    for \( n \) vertices and
    \( m = \Theta( n^2 )\) edges (dense case). 
    In this paper, we give
    a solution to the sparse case \( m = n(1 + \mu n^{-1/3}) \)
    with \( \mu \) bounded or going to $-\infty$
    (\cref{theorem:asymptotics:dags}).
    The first case exhibits a phase transition
    reminiscent of directed graphs (see~\cite{luczak2009critical}).
    In the second case, when \( 0 < \lim \frac{m}{n} < 1 \), we have
    \( \mathbb P( \)digraph is acyclic\()= e^{-m/n}(1-\frac{m}{n}) \).

    \paragraph*{Exact and asymptotic enumeration.} In 1973, Robinson~\cite{robinson1973counting}
    obtained his beautiful formula for the number \( \fdag_{n,m} \) of labeled
    DAGs with \( n \) vertices and \( m \) edges
    \[
        \fdag_{n,m} = n! [z^n w^m] \dfrac{(1 + w)^{n \choose 2}}
        {\sum_{n \geq 0} (1 + w)^{- {n \choose 2}} \frac{(-z)^n}{n!}},
    \]
    and developed a framework for the 
enumeration of digraphs whose strongly connected
    components
    (called \emph{strong components} in the following)
    belong to a given family of allowed strongly connected digraphs.
    This yielded the asymptotics of dense DAGs
    in~\cite{bender1986asymptotic}.
    The structure of random DAGs has been studied
    in~\cite{liskovets1976number,mckay1989shape,gessel1996counting}.

    We say that a digraph is \emph{elementary} if all its strong
    components are either isolated vertices or cycles.
    In~\cite{luczak1990phase} and
    \cite{luczak2009critical} it was shown that
    if the parameter \( c=\frac{m}{n} \) is less than one, then a 
    digraph is elementary asymptotically almost surely.
    More precisely, this happens when a digraph has \( n \) vertices and \( m =
    n(1 + \mu n^{-1/3}) \) edges, as \( \mu \to -\infty \) with \( n \).
    Other interesting structural results around
    the phase transition point are available
    in~\cite{pittel2017birth,goldschmidt2019scaling}. 
    The authors of~\cite{goldschmidt2019scaling} show that the
    strong components have asymptotically almost surely cubic kernels, \ie the sum
    of the degrees of each of its nodes is at most three with high probability.
    This means that these cubic kernels play an analogous role as the classical
    cubic kernels in a random graphs, see~\cite{janson1993birth}.

    A forthcoming independent approach of~\cite{sparserandomacyclicdigraphs} in
    the analysis of asymptotics of DAGs (manuscript to appear), is similar in
    spirit to the tools used in~\cite{flajolet2004airy} 
    and relies on a bivariate singularity analysis of the generating function
    of DAGs.
    Their technique promises to unveil sparse DAGs asymptotics,
    covering as well the case
    where the ratio of the numbers of edges and vertices is bounded,
    but greater than $1$ (the \emph{supercritical case}).

    \paragraph*{Our contribution.}
    Typically, the analysis of graphs is technically easier when loops
    and multiple edges are allowed,~\cite{janson1993birth}:
    an adaptation of the symbolic techniques to the case of
    simple graphs becomes rather a technical, but not a conceptual
    obstacle.
    In~\cite{panafieu2016analytic} and \cite{collet2018threshold}
    the dedicated \emph{patchwork} concept is introduced allowing to handle
    this difficulty.
    The same principle concerns directed graphs. Nevertheless, in the
    current paper we consider the case of \emph{simple digraphs} where loops
    and multiple edges are forbidden. In our model, however, the cycles of
    size 2 are allowed, because it is natural to suppose that for each two
    vertices \( i \) and \( j \) both directions are allowed. The analysis
    of simple digraphs is technically heavier than the analysis of
    multidigraphs, but we prefer to demonstrate explicitly that such an
    application is indeed possible.

    Firstly, we transform the generating function of DAGs so that it can
    be decomposed into an infinite sum. Each of its summands is analysed
    using a new bivariate semi-large powers lemma which is a generalisation
    of~\cite{banderier2001random}. We discover (in the above notations) that
    the first term of this infinite expansion is dominating in the
    \emph{subcritical} case, \ie when \( \mu \to -\infty \); in the case
    when \( \mu \) is bounded (the \emph{critical} case), all the terms give
    contributions of the same order.
    Next, using the symbolic tools for directed graphs
    from~\cite{de2019symbolic}, we express the generating function of
    elementary digraphs and apply similar tools to obtain explicitly the
    phase transition curve in digraphs, that is, the probability that a
    digraph is elementary, as a function of \( \mu \).

    \paragraph*{Related studies.}
        An alternative exact formula for the number of DAGs (without taking into
        account the number of edges) is obtained in~\cite{mckay2004acyclic}, where
        the authors show that the number of DAGs with \( n \) labeled vertices is
        equal to the number of \( n \times n \) $(0, 1)$-matrices whose eigenvalues
        are positive real numbers.

        It is worth mentioning that the exact enumeration formula for acyclic
        digraphs does not give an immediate receipt for uniform generation, as
        the formula includes inclusions-exclusions. The question of random
        generation was settled in~\cite{kuipers2015uniform}.

    Analytic techniques, largely covered in~\cite{flajolet2009analytic}, are
    efficient for asymptotic analysis, because the coefficient 
    extraction operation is naturally expressed through Cauchy formula.
    A recent study~\cite{greenwood2018asymptotics} is dealing with bivariate
    algebraic functions. In their case, a combination of two Hankel contours,
    necessary for careful analysis, can have a complicated mutual configuration
    in two-dimensional complex space, so a lot of details needs to be accounted
    for. Our approach is close to theirs, while we try to avoid the mentioned
    difficulty in our study. The principle idea behind our bivariate semi-large
    powers lemma is splitting a double complex integral into a product of two
    univariate ones.

        EGF of dags is an inverse deformed exponential function.  This
        function is studied in \cite{wang2018zeros}, see references
        therein.  Special functions very similar to EGF of DAGs arise
        with connection to Airy function, see \cite{banderier2001random}
        \cite{flajolet2004airy}.
        The behaviour of the transition curve we obtain for DAGs is
        somewhat similar to the result of~\cite{herve2011random}.

    \paragraph*{Structure of the paper.}
    In~\cref{section:exact:expressions} we present new exact reformulations of the numbers of DAGs and elementary digraphs, which are later used
    in~\cref{section:asymptotic:analysis} to ease the asymptotic analysis.

\end{section}

\begin{section}{Exact expressions using generating functions}
    \label{section:exact:expressions}

    Consider the following model of graphs and directed graphs.
    A graph $G$ is characterized by
    its set $V(G)$ of labeled vertices
    and its set $E(G)$ of unoriented unlabeled edges.
    Loops and multiple edges are forbidden.
    The numbers of its vertices and edges
    are denoted by $n(G)$ and $m(G)$.
    An $(n,m)$-graph (or digraph) is a graph (or digraph) with $n$ vertices and $m$ edges.

    We consider digraph without loops, such that from any vertex
    \( i \) to any vertex \( j \) there can be at most one directed edge.
    Therefore, two edges can link the same pair of vertices
    only if their orientations are different.

    \begin{subsection}{Exponential and graphic generating functions}

        Two helpful tools in the study of graphs and directed graphs are the
    \emph{exponential generating function} (EGF) and \emph{graphic generating
    function} (GGF).
        The EGF $\eF(z, w)$
        and the GGF $\gF(z, w)$
        associated to a graph or digraph family $\fF$ are defined as
        \[
          \eF(z,w) =
          \sum_{G \in \fF} w^{m(G)} \frac{z^{n(G)}}{n(G)!},
          \quad
          \gF(z,w) =
          \sum_{G \in \fF} w^{m(G)}
            \frac{z^{n(G)}}{n(G)! (1 + w)^{ {m(G) \choose 2} }}.
        \]
        The total numbers of $(n,m)$-graphs and $(n,m)$-digraphs
        are $\binom{n(n-1)/2}{m}$ and $\binom{n(n-1)}{m}$.
        The classical counting expression for directed acyclic graphs is
        attributed to Robinson~\cite{robinson1973counting}.
        The EGF \( \eG(z, w) \) of all graphs and GGF of directed acyclic graphs
        \( \gdag(z, w) \) are given by
        \begin{equation}
            \label{eq:graphs:dags}
            \eG(z,w) = \sum_{n \geq 0} (1+w)^{\binom{n}{2}} \frac{z^n}{n!},
            \quad
            \gdag(z, w) = \dfrac{1}{
                \sum\limits_{n \geq 0} (1+w)^{-\binom{n}{2}} \frac{(-z)^n}{n!}
            }.
        \end{equation}

        We can reuse the EGF of graphs~\eqref{eq:graphs:dags} to obtain an
        alternative expression for the number of $(n,m)$-DAGs \( \fdag_{n,m} \):
        \begin{equation}
            \label{th:dags_exact}
            \fdag_{n,m} =
            n! [z^n w^m]
            \frac{(1+w)^{\binom{n}{2}}}{\eG \left(-z, -\frac{w}{1+w} \right)}.
        \end{equation}

        Before considering various digraph families, we need to recall the
        classical generating functions of simple graph families, namely the rooted
        and unrooted labeled trees and unicycles.
        A \emph{unicycle} is a connected graph that has the same numbers of vertices
        and edges. Hence, it contains exactly one cycle.

        \begin{proposition}[{\cite{janson1993birth}}] \label{th:trees_and_unicycles}
        The EGFs $T(z)$ of rooted trees,
        $U(z)$ of trees and
        $V(z)$ of unicycles
        are characterized by the relations
        \[
            T(z) = z e^{T(z)}, \quad
            U(z) = T(z) - \frac{T(z)^2}{2}, \quad
            V(z) = \frac{1}{2} \log \left( \frac{1}{1 - T(z)} \right)
            - \frac{T(z)}{2} - \frac{T(z)^2}{4}.
        \]
        \end{proposition}

        The \emph{excess} of a graph (not necessarily connected)
        is defined as the difference between
        its numbers of edges and vertices. For example, trees have excess \( -1
        \), while unicycles have excess \( 0 \).
        The bivariate EGFs of graphs of excess $k$ can be
        obtained from their univariate EGFs by substituting \( z \mapsto zw \)
        and multiplying by \( w^k \). In particular, 
        \( T(z, w) = T(zw)/w \), \( U(z, w) = U(zw)/w \), \( V(z, w) = V(zw) \).

        We say that a graph is \emph{complex} if
        all its connected components have a positive excess.
        The EGF of complex graphs of excess $k$ is
        \[
            \eComplex_k(z) = [y^k] \eComplex(z/y,y).
        \]
        %
        %
        It is known (see~\cite{janson1993birth}) that a
        complex graph of excess \( r \)
        is reducible to a \emph{kernel} (multigraph of minimal degree at least $3$)
        of same excess,
        by recursively removing vertices of degree $0$ and $1$
        and fusioning edges sharing a degree $2$ vertex.
        The total weight of \emph{cubic} kernels (all degrees equal to $3$) of excess $r$
        is given by~\eqref{eq:complex:cf}.
        They are central in the study of large critical graphs,
        because non-cubic kernels do not typically occur.
        %

        \begin{proposition}[{\cite[Section 6]{janson1993birth}}] \label{th:complex}
        For each \( r \geq 0 \) there is a polynomial \( P_r(T) \) such that
        \begin{equation}
            \label{eq:complex:cf}
            \eComplex_r(z) =
            e_{r}\cfrac{T(z)^{5r}}{(1-T(z))^{3r}}
            +
            \dfrac{P_r(T(z))}{(1 - T(z))^{3r - 1}},
            \quad
            \text{where}
            \quad
            e_r = \frac{(6r)!}{2^{5r} 3^{2r} (2r)! (3r)!}. 
        \end{equation}
        \end{proposition}

        Since any graph can be represented as a set of unrooted trees, unicycles
        and a complex component (whose excess is denoted by \( k \) below) the EGF of graphs is equal to
        \begin{equation}
        \label{th:graph_decomposition}
          \eG(z,w) =
          e^{U(z w) / w}
          e^{V(z w)}
          \sum_{k \geq 0}
          \eComplex_k(zw) w^k.
        \end{equation}

    \end{subsection}

    \begin{subsection}{Exact expression for directed acyclic graphs}
        In order to obtain the asymptotic number of DAGs, we need a
        decomposition different from~\eqref{eq:graphs:dags}.
        For comparison, in the expression~\eqref{th:graph_decomposition} the
        first summand is asymptotically dominating in the case of subcritical
        graphs. Inside the critical window, all the summands
        of~\eqref{th:graph_decomposition} give a contribution of the
        same asymptotic order.

        \begin{lemma} \label{th:dags_new}
        The number $\fdag_{n,m}$ of $(n,m)$-DAGs is equal to
        \begin{multline*}
            n!^2 \sum_{t \geq 0}
            [z_0^n z_1^n]
            \frac{(U(z_0) + U(z_1))^{2n-m+t}}{(2n-m+t)!}
            \frac{e^{U(z_1) + V(z_0)}}{e^{ V(z_1)}}
            [y^t]
            \frac{
                \sum_{j \geq 0} \eComplex_j(z_0) y^j
            }{
                \sum\limits_{k \geq 0} \eComplex_k(z_1)
                \left(- \frac{y}{1+y} \right)^k
            }
          \frac{
              1
          }{(1+y)^n}.
        \end{multline*}
        \end{lemma}

        \begin{proof}
        Since \( (1 + w)^{ n \choose 2} \) is the generating function of
        graphs with \( n \) vertices, we can replace
        $(1+w)^{\binom{n}{2}}$ with $n! [z^n] \eG(z,w)$
            in~\eqref{th:dags_exact}.
        Injecting the expression of $\eG(z,w)$
        from~\eqref{th:graph_decomposition}
        in the resulting formula with \( z \mapsto -z_1 \) and
        \( w \mapsto -\frac{w}{1 + w} \),
        we obtain (see also Remark~\ref{remark:two:graphs} for more intuitions)
        \[
            \fdag_{n,m} =
            n!^2 [z_0^n z_1^n w^m]
            \frac{
                e^{U(z_0 w)/w + V(z_0 w)} \sum_{j \geq 0} \eComplex_j(z_0 w) w^j
            }{
                e^{- U(\frac{z_1 w}{1+w}) \frac{1+w}{w}
                + V(\frac{z_1 w}{w+1})}
                \sum_{k \geq 0} \eComplex_k
                \left(
                    \frac{z_1 w}{1+w}
                \right)
                \left(
                    - \frac{w}{1+w}
                \right)^k
            }.
        \]
        The change of variables $(z_0, z_1, w) \mapsto \left( \frac{z_0}{y},
        \frac{1+y}{y} z_1, y \right)$ are applied, which results in
        \[
            \fdag_{n,m} =
            n!^2 [z_0^n z_1^n y^{m-2n}]
            e^{(U(z_0) + U(z_1))/y}
            e^{U(z_1)}
            \frac{\sum_{j \geq 0} \eComplex_j(z_0) y^j}
              {\sum_{k \geq 0} \eComplex_k(z_1) \left(- \frac{y}{1+y} \right)^k}
            \frac{e^{V(z_0) - V(z_1)}}{(1+y)^n}.
        \]
        We finish the proof by extracting the coefficient $[y^{m-2n}]$.
    \end{proof}

    \begin{remark}
        \label{remark:two:graphs}
  The number of pairs $(G_0, G_1)$ of graphs,
  each on $n$ vertices, having a total of $m_0 + m_1 = m$ edges,
  is
  \(
    n!^2 [z_0^n z_1^n w^m] G(z_0, w) G(z_1, w).
  \)
  Working as in the previous proof leads to
  \[
    n!^2 \sum_{t \geq 0} [z_0^n z_1^n w^m]
    \frac{\left(U(z_0) + U(z_1)\right)^{2n-m+t}}{(2n-m+t)!}
    e^{V(z_0) + V(z_1)}
    \sum_{j \geq 0} \eComplex_j(z_0) y^j
    \sum_{k \geq 0} \eComplex_k(z_1) y^k.
  \]
  which looks and behaves (when $m/n$ stays smaller than
   or close to $1$) like the expression for $\fdag_{n,m}$ from the last lemma.
  This motivates the following intuition.
  Typically, those two graphs should share the $m$ edges more or less equally.
  Thus, when $m/n$ is close to $1$, $m_0/n$ and $m_1/n$ should be close to $1/2$,
  so $G_0$ and $G_1$ will exhibit critical graph structure. For a smaller ratio $m/n$,
  $G_0$ and $G_1$ will behave like subcritical graphs, containing only trees and unicycles.
  This heuristic explanation for the critical density for dags
  guides our analysis in the rest of the paper.
    \end{remark}

    \end{subsection}

    \begin{subsection}{Exact expression for elementary digraphs}
        As we discovered in our previous paper~\cite{de2019symbolic}, and which
        was also pointed in a different form in~\cite{robinson1973counting},
        the graphic generating function of the
        family of digraphs whose connected components belong to a given set \(
        \family{S} \) with the EGF \( S(z, w) \), is
        given by
        \begin{equation}
            \label{eq:subcritical:spec}
            \graphic{E}(z, w) = \dfrac{1}{e^{-\eS(z, w)} \hadamard_z \gset(z,
            w)}, \quad \text{where} \quad
            \gset(z, w) = \sum_{n \geq 0} \dfrac{z^n}{n! (1 + w)^{n \choose 2}},
        \end{equation}
        and \( \hadamard_z \) is the exponential Hadamard product,
        characterized by $\sum_n a_n \frac{z^n}{n!} \hadamard_z \sum_n b_n \frac{z^n}{n!} = \sum_n a_n b_n \frac{z^n}{n!}$. 
        \( \gset(z, w) \) is the GGF of sets of
        isolated vertices.
        In particular, for the case of \emph{elementary digraphs}, \ie the
        digraphs whose strong components are isolated vertices or cycles of
        length \( \geq 2 \) only, the EGF of \( \family{S} \) is given by
        \[
            \eS(z, w) = z + \ln \dfrac{1}{1 - zw} - zw.
        \]
        In order to expand the Hadamard product, we develop the exponent \(
        e^{-\eS(z, w)} \) and apply the simplification rule
        \( az e^{az} \hadamard_z F(z) = \left.z \frac{d}{dz}F(z)\right|_{z
        \mapsto az} \).
        After developing the exponent
        and expanding the Hadamard product
        we obtain a very simple expression, namely
        \begin{equation}
            \label{eq:subcritical:digraphs}
            \graphic{E}(z, w)
            = \dfrac{1}{(1 - zw)e^{-z(1 - w)} \hadamard_z \gset(z, w)}
            =
            \left.
            \dfrac{1}{
                \gset(-z, w)
                + z\frac{w}{1 - w}
                \frac{d}{dz} \gset(-z, w)
            }
            \right|_{z\ \mapsto (1 - w)z}.
        \end{equation}

        The following lemma is a heavier version of this expression. One of the
        reasons behind its visual complexity is the
        choice of the simple digraphs instead of multidigraphs; however, during
        the asymptotic analysis, most of the decorations corresponding to simple
        digraphs are going to disappear.

        \begin{lemma} The number \( \family{ED}_{n,m} \) of \( (n,m) \) elementary
            digraphs is equal to
            \label{lemma:number:subcritical:dags}
            \begin{multline*}
                \family{ED}_{n,m}
                = n!^2
                \sum_{t \geq 0}
                [z_0^n z_1^n]
                \dfrac{
                    (
                    U(z_0) + U(z_1)
                    )^{2n -m + t}
                }{
                    (2n - m + t)!
                }
                [y^t]
                e^{\frac{2}{1 - y} U(z_1) + V(z_0) - V(z_1)}
                \left(
                    \dfrac{1 - y}{1 + y}
                \right)^n
                \\
                \times
                \dfrac{\sum_{j \geq 0} \eComplex_j(z_0) y^j}
                {\sum\limits_{k \geq 0}
                  \left[
                      \eComplex_k(z_1) \left(
                          1
                          - \frac{1 + y}{1 - y}
                          (
                            T(z_1) - V^\bullet(z_1)
                          )
                      \right)
                      +
                      \frac{1 + y}{1 - y}
                      \eComplex^\bullet_k(z_1)
                  \right]
                  \left(- y\frac{1-y}{1+y} \right)^k}
                  ,
            \end{multline*}
            where
            \(
                \eComplex^\bullet_r(z) = z \frac{d}{dz}
                \eComplex^\bullet_r(z)
            \) and \(
                V^\bullet(z) = z \frac{d}{dz}V(z).
            \)
        \end{lemma}

        \begin{proof}
            Let us denote \( v = \frac{w}{1 + w} \).
            Using the already mentioned representation
            \[
                \gset(-z, w)
                = \eG\left(-z, -\frac{w}{1 + w}\right)
                = e^{-U(zv) / v
                    + V(zv)}
                    \sum_{r \geq 0}\eComplex_r(zv)
                    \left(
                        -v
                    \right)^r,
            \]
            and by replacing \( (1 + w)^{n \choose 2} \) with 
the generating
            function of graphs with \( n \) vertices as in the proof
            of~\cref{th:dags_new}, we can write the denominator
            of~\eqref{eq:subcritical:digraphs} prior to substitution \( z
            \mapsto (1 - w)z \) as
            \begin{multline*}
                \gset(-z, w) + z \dfrac{w}{1 - w} \dfrac{d}{dz}
                \gset(-z, w)
                =
                e^{-U(zv) / v + V(zv)}
                \left[
                    \sum_{r \geq 0} \eComplex_r(zv) (-v)^r
                    \right.
                    \\
                    \left.
                    -
                    \dfrac{1 + w}{1 - w}
                    \left(
                    \left[
                    T(zv)
                    -
                    V^{\bullet}(zv)
                    \right]
                    \sum_{r \geq 0} \eComplex_r(zv) (-v)^r
                    +
                    \sum_{r \geq 0} \eComplex^\bullet_r(zv) (-v)^r
                    \right)
                \right],
            \end{multline*}
            Next, the change of variables
            \( (z_0, z_1, w) \mapsto \left(
                \dfrac{z_0}{y},
                \dfrac{1 + y}{1 - y} \dfrac{z_1}{y},
                y
            \right)\) yields
            \begin{multline*}
                \family{ED}_{n,m}
                = n!^2
                [z_0^n z_1^n y^{m-2n}]
                e^{ (U(z_0) + U(z_1))/y}
                e^{\frac{2}{1 - y} U(z_1) + V(z_0) - V(z_1)}
                \left(
                    \dfrac{1 - y}{1 + y}
                \right)^n
                \\
                \times
                \dfrac{\sum_{j \geq 0} \eComplex_j(z_0) y^j}
                {\sum\limits_{k \geq 0}
                  \left[
                      \eComplex_k(z_1) \left(
                          1
                          - \frac{1 + y}{1 - y}
                          (
                            T(z_1) - V^\bullet(z_1)
                          )
                      \right)
                      +
                      \frac{1 + y}{1 - y}
                      \eComplex^\bullet_k(z_1)
                  \right]
                  \left(- y\frac{1-y}{1+y} \right)^k}
                  .
            \end{multline*}
            The proof is finished by extracting the coefficient \( [y^{m - 2n}]
            \).
        \end{proof}

    \end{subsection}
\end{section}

\begin{section}{Asymptotic analysis}
    \label{section:asymptotic:analysis}
    \begin{subsection}{Bivariate semi-large powers lemma}

        The typical structure of critical random graphs
        can be obtained by application of the \emph{semi-large powers Theorem}
        \cite[Theorem~IX.16, Case (ii)]{flajolet2009analytic}.
        Since DAGs behave like a superposition of two graphs
        (see~\cref{remark:two:graphs}),
        we design a bivariate variant of this theorem.

        \begin{lemma}
        \label{lemma:bdg} 
        Consider two integers $n$ and $m$ going to infinity,
        such that $m = n (1 + \mu n^{-1/3})$
        with $\mu$ either staying in a bounded real interval, or \( \mu \to
        -\infty \) while \(\liminf\limits_{n \to \infty} m/n > 0\); let the function $F(z_0, z_1)$ be analytic on the open torus of
        radii $(1, 1)$
        \( \{ z_0, z_1 \in \mathbb C \colon |z_0| < 1, |z_1| < 1 \} \)
        and continuous on its closure,
        and let $r_0$ and $r_1$ be two real values,
        then the following asymptotics holds as \( n \to \infty \)
        \begin{multline}
            \label{eq:bivariate:sp:lemma}
            [z_0^n z_1^n] \left( U(z_0) + U(z_1) \right)^{2n-m}
            \frac{F(T(z_0), T(z_1))}{(1-T(z_0))^{r_0} (1-T(z_1))^{r_1}}
            \\ \sim\
            \frac{e^{2n}}{4}
            \left( \frac{3}{n} \right)^{(4 - r_0 - r_1)/3}
            F\left(\dfrac{m}{n}, \dfrac{m}{n} \right)
            H\left(\frac{3}{2}, \frac{r_0}{2}, - \frac{3^{2/3}}{2} \mu \right)
            H\left(\frac{3}{2}, \frac{r_1}{2}, - \frac{3^{2/3}}{2} \mu \right),
        \end{multline}
        where the function $H(\lambda, r, x)$ is defined as
        \(
            \frac{1}{\lambda}
            \sum_{k \geq 0}
            \Gamma \left( \frac{\lambda + r - k - 1}{\lambda} \right)^{-1}
            \frac{(-x)^k}{k!}.
        \)
        \end{lemma}

        \begin{remark}
        \label{remark:H}
        A direct computation shows that \( H(\cdot, \cdot,
        \cdot) \) from~\eqref{eq:bivariate:sp:lemma} can be expressed as
        \[
            H\left(
                \dfrac32, r,
                - \dfrac{3^{2/3}}{2} \mu
            \right)
            =
            \dfrac{2}{3}
            e^{\mu^3/6} 3^{(2r+1)/3}
            A(2r, \mu),
        \]
        where the function \( A(y, \mu) \) is defined in~\cite{janson1993birth}
        as
        \[
            A(y, \mu) = \dfrac{e^{-\mu^3/6}}{3^{(y+1)/3}}
            \sum_{k \geq 0}
            \dfrac{
                \big(
                    \frac12 3^{2/3} \mu
                \big)^k
            }{
                k! \Gamma\big(
                    \frac{y+1 - 2k}{3}
                \big)
            }
            \text{\ and satisfies\ }
            \lim_{\mu \to -\infty} A(y, \mu) |\mu|^{y - 1/2}
            = \dfrac{1}{\sqrt{2 \pi}}
            .
        \]
        We provide only the proof of the harder case when \( \mu \) is bounded.
        In order to adapt the proof of~\cref{lemma:bdg} to the case \( \mu \to
        -\infty \), a simpler saddle-point bound can be used.
        \end{remark}

        \begin{proof}[Proof of~\cref{lemma:bdg}]
            The first step is to represent the coefficient extraction operation
            from~\eqref{eq:bivariate:sp:lemma} as a double complex integral,
            using Cauchy formula, and to approximate this double integral with a
            product of two complex integrals.
            We start with the Puiseux expansion of the EGF of rooted labeled
            trees \( T(z) \) and unrooted labeled trees \( U(z) = T(z) -
            \frac{T^2(z)}{2}\):
            \begin{align}
              \label{eq:newton:puiseux:T}
              T(z) &= 1 - \sqrt{2} \sqrt{1 - e z}
              + \frac{2}{3} (1 - e z) + \bigO(1 - e z)^{3/2}, \\
              U(z) &= \frac{1}{2} - (1 - e z)
              + \frac{2^{3/2}}{3} (1 - e z)^{3/2} + \bigO(1 - e z)^2.
              \label{eq:newton:puiseux:U}
          \end{align}
            Applying Cauchy's integral theorem, we rewrite the coefficient
            extraction~\eqref{eq:bivariate:sp:lemma} in the form
            \[
                \frac{1}{(2 i \pi)^2}
                \oint \oint
                \left(
                    U(z_0) + U(z_1)
                \right)^{2n-m}
                \frac
                    {F(T(z_0), T(z_1))}
                    {(1-T(z_0))^{r_0} (1-T(z_1))^{r_1}}
                \frac{d z_0}{z_0^{n+1}} \frac{d z_1}{z_1^{n+1}}.
            \]
            \begin{table}[H]
            \begin{minipage}{0.6\textwidth}
            In order to accomplish the separation of the integrals, we represent
            the term
            \(
                \left(
                    U(z_0) + U(z_1)
                \right)^{2n-m}
            \)
            as the exponent of the logarithm, and evaluate the leading terms in
            Newton--Puiseux expansion of the logarithm
            \[
                \left(
                    U(z_0) + U(z_1)
                \right)^{2n-m}
                =
                e^{
                    (2n-m) \log \left(
                        U(z_0) + U(z_1)
                    \right)
                }.
            \]
            By plugging the leading terms of
            \( U(z_0) \) and \( U(z_1) \)
            from~\eqref{eq:newton:puiseux:U}
            into the previous expression
            and developing the logarithm around \( z_0 = z_1 = e^{-1} \),
            we notice that the leading powers of \( (1 - ez_0) \) and
            \( (1 - ez_1) \)
            contain only the exponents \( \{0, 1, \frac{3}{2}\} \), and thus,
            asymptotically, no products need to be taken into account,
            see~\cref{table:product:form}.
            \end{minipage}\quad\quad
            \begin{minipage}{0.35\textwidth}
                \centering
                \begin{tabular}{cccc}
                    & 0 & 1 & \( \frac{3}{2} \vphantom{\Big|} \) \\
                    \hline
                    \( 0 \vphantom{\Big|} \) & \checkmark
                            & \checkmark & \checkmark \\
                    \( 1 \vphantom{\Big|} \)
                            & \checkmark & -- & -- \\
                    \( \frac{3}{2} \vphantom{\Big|} \)
                            & \checkmark & -- & -- \\
                    \hline
                \end{tabular}
                \caption{Contributing exponents of \( (1 - e z_0) \) and \( (1 -
                    ez_1) \) for bivariate semi-large powers
                lemma}
                \label{table:product:form}
            \end{minipage}%
            \end{table}

            A further step is to inject \( m = n + \mu n^{2/3} \), 
            \( 1 - ez_0 = \alpha_0 n^{-2/3} \), and \( 1 - ez_1 = \alpha_1
            n^{-2/3} \), where \( \alpha_0, \alpha_1 \in \mathbb C \).
            By using expansion~\eqref{eq:newton:puiseux:T} in order to
            approximate the terms \( (1 - T(z_0)) \) and \( (1 - T(z_1)) \),
            we rewrite the answer in the form
            \begin{multline*}
                \left( \frac{n}{2^{3/2}} \right)^{(r_0 + r_1 - 4)/3}
                \frac{e^{2n}}{4 (2 i \pi)^2}
                F\left(1,1\right) \times \\
                \oint \oint
                e^{
                  \mu (\alpha_0 + \alpha_1)
                  + \frac{2^{3/2}}{3} (\alpha_0^{3/2} + \alpha_1^{3/2})
                  + \bigO \left( n^{-1/3} \right)
                  }
                \left(1 + \bigO(n^{-1/3}) \right)
                \frac{d \alpha_0}{\alpha_0^{r_0/2}}
                \frac{d \alpha_1}{\alpha_1^{r_1/2}}.
            \end{multline*}
            After removal of the negligible terms,
            a product of integrals is obtained
            \[
              F(1, 1)
              \left( \frac{n}{2^{3/2}} \right)^{(r_0 + r_1 - 4)/3}
              \frac{e^{2n}}{4}
              \frac{1}{2 i \pi} \oint
              e^{\mu \alpha_0 + \frac{2^{3/2}}{3} \alpha_0^{3/2}}
              \frac{d \alpha_0}{\alpha_0^{r_0/2}}
              \times
              \frac{1}{2 i \pi} \oint
              e^{\mu \alpha_1 + \frac{2^{3/2}}{3} \alpha_1^{3/2}}
              \frac{d \alpha_1}{\alpha_1^{r_1/2}}.
            \]
            Each of the integrals can be evaluated similarly as
            in~\cite[Theorem~IX.16, Case (ii)]{flajolet2009analytic}: in order
            to evaluate such integral, a variable change
            \( u = - \frac{2^{3/2}}{3} \alpha^{3/2} \) is applied, and the
            integral is expressed as an infinite sum using a Hankel contour
            formula for the Gamma function:
            \[
                \frac{1}{2 i \pi} \oint
                e^{\frac{3^{2/3}}{2} \mu u^{2/3}}
                e^{- u}
                \left( \frac{3^{2/3} u^{2/3}}{2} \right)^{-\frac{1+r}{2}}
                \frac{du}{\sqrt{2}}
                =
                \frac{2^{r/2}}{3^{(1+r)/3}}
                \sum_{k \geq 0}
                \frac{ \left( \frac{3^{2/3} \mu}{2} \right)^k}{k!}
                \frac{1}{2 i \pi} \oint
                u^{\frac{2 k - 1 - r}{3}}
                e^{-u} du.
            \]
        \end{proof}
    \end{subsection}

    \begin{subsection}{Asymptotic analysis of directed acyclic graphs}
        Since we are going to apply~\cref{lemma:bdg} to each of the terms
        of the infinite sum of~\cref{th:dags_new}, it is useful to introduce the
        following notation
        \[
            s_r^+(\mu) = H\left(
                \dfrac32, \dfrac{3r}{2} + \dfrac{1}{4}
                , - \dfrac{3^{2/3}}{2} \mu
            \right),
            \quad
            s_r^-(\mu) = H\left( 
                \dfrac32, \dfrac{3r}{2} - \dfrac{1}{4}
                , - \dfrac{3^{2/3}}{2} \mu
            \right);
        \]
        \[
            S^+(y; \mu) = \sum_{r \geq 0} s_r^+(\mu) y^r,
            \ \
            S^-(y; \mu) = \sum_{r \geq 0} s_r^-(\mu) y^r,
            \ \
            E(y) = \sum_{r \geq 0} e_r y^r,
            \ \
            e_r^{(-1)} = [y^r] \dfrac{1}{E(-y)},
        \]
        where \( e_r \) is given by~\cref{th:complex}.
        This notation will be used throughout the next two sections.

        \begin{theorem}
            \label{theorem:asymptotics:dags}
            When \( m = n(1 + \mu n^{-1/3}) \) and \( \mu \) either stays in a
            bounded real interval, or \( \mu \to -\infty \)
            while \(\liminf m/n > 0 \) as \( n \to \infty \),
            \[
                \mathbb P( (n,m)\text{-digraph is acyclic})
                \sim
                \frac{3^{5/6}}{2}
                \frac{e^{\frac{m}{n} - \frac{\mu^3}{6}} \sqrt{2 \pi}}{n^{1/3}}
                \sum_{q \geq 0}
                3^{-q} e^{(-1)}_q
                s_q^-(\mu).
            \]
            In particular, for the sparse case where $\mu \to -\infty$
            (which covers $\limsup m/n < 1$),
            \[
                \mathbb P( (n,m)\text{-digraph is acyclic})
                \sim
                e^{m/n} (1 - m/n).
            \]
        \end{theorem}

        \begin{proof}
            In order to apply~\cref{lemma:bdg} (bivariate semi-large powers), we
            develop the coefficient operator \( [y^t] \) in~\cref{th:dags_new}
            using the approximation of \( \eComplex_k(T) \)
            from~\cref{th:complex}
            and drop the terms that give negligible contribution:
            \begin{multline*}
                [y^t]
                \frac{
                    \sum_{j \geq 0} \eComplex_j(z_0) y^j
                }{
                    \sum\limits_{k \geq 0} \eComplex_k(z_1)
                    \left(- \frac{y}{1+y} \right)^k
                }
                \dfrac{1}{(1 + y)^n}
                \\=
                \!\!\!\!
                \sum_{p + q + r = t}
                \frac{e_p T(z_0)^{2p}}{(1-T(z_0))^{3p}}
                \frac{e^{(-1)}_q T(z_1)^{2q}}{(1-T(z_1))^{3q}}
                [y^r] \frac{1}{(1+y)^n}
                (
                1 + \bigO(1 - T(z_0))
                ) (
                1 + \bigO(1 - T(z_1))
                ).
            \end{multline*}
            Then we apply~\cref{lemma:bdg} and the approximation $(2n-m+t)! \sim
            (2n-m)! n^t$ to obtain
            \begin{align*}
                \fdag_{n,m}
                &=
                n!^2
                \sum_{t \geq 0}
                [z_0^n z_1^n]
                \sum_{p + q + r = t}
                \frac{(U(z_0) + U(z_1))^{2n-m+t}}{(2n-m+t)!}
                e^{U(z_1)}
                \frac{e_p T(z_0)^{2p}}{(1-T(z_0))^{3p+1/2}}
                \\
                & \phantom{=\ } \times 
                \frac{e^{(-1)}_q T(z_1)^{2q}}{(1-T(z_1))^{3q-1/2}}
                [y^r] \frac{1}{(1+y)^n}
                ( 1 + \bigO(1-T(z_0)))
                ( 1 + \bigO(1-T(z_1)))
                \\
                &=
                \dfrac{n!^2}{(2n - m)!}
                \sum_{t \geq 0}
                \sum_{p + q + t = t}
                \dfrac{e^{m/2n}}{n^t}
                e_p e_q^{(-1)}
                \dfrac{e^{2n}}{4}
                \left( \frac{3}{n} \right)
                ^{(4 - (3p+1/2) - (3q-1/2))/3}
                \\
                &  \phantom{=\ } \times 
                s_p^+(\mu)
                s_q^-(\mu)
                [y^r] \dfrac{1}{(1 + y)^n}.
            \end{align*}
        The power of $n$ in the sum is $n^{-4/3-r}$, and the sum over $r$ of
        $n^{-r} [y^r] (1+y)^{-n}$ is equal to $(1+1/n)^{-n}$ and converges to
        $e^{-m/n}$.  Finally, the sums over $p$ and $q$ are decoupled and we
        obtain
        \[
            \fdag_{n,m} \sim
            \dfrac{n!^2}{(2n - m)!}
            \frac{e^{2n-m/2n}}{n^{4/3}}
            \frac{3^{4/3}}{4}
            \sum_{p \geq 0} 3^{-p} e_p s_p^+(\mu)
            \sum_{q \geq 0} 3^{-q} e_q^{(-1)} s_q^-(\mu).
        \]
        The sum over $p$ admits a closed expression $\sqrt{\dfrac{2}{3 \pi}}
        e^{\mu^3/6}$ (see~\cref{remark:H} and~\cite[Section 14]{janson1993birth}).
        Applying Stirling's formula, we can rescale the asymptotic number of
        DAGs by the total number of digraphs:
        \[
          \fdag_{n,m} \sim
          \binom{n(n-1)}{m}
          \frac{3^{5/6}}{2}
          \frac{e^{\frac{m}{n} - \frac{\mu^3}{6}} \sqrt{2 \pi}}{n^{1/3}}
          \sum_{q \geq 0}
          3^{-q} e^{(-1)}_q
          s_q^-(\mu).
        \]
        This gives the main statement.
        To obtain the sparse case, we need to use the fact that
        when \( \mu \to -\infty \), the first summand of the sum over \( q \) is
        dominating, and therefore, this sum is asymptotically equivalent to
        \( \sqrt{\frac{2}{\pi}} \frac{e^{\mu^3/6}}{|\mu|^{-1}} 3^{-5/6} \)
        (see~\cite[Equation (10.3)]{janson1993birth}).
        \end{proof}

    \end{subsection}
    \begin{subsection}{Asymptotic analysis of elementary digraphs}
        \begin{theorem}
            \label{theorem:asymptotics:elementary}
            When \( m = n(1 + \mu n^{-1/3}) \) and \( \mu \) either stays in a
            bounded real interval, or \( \mu \to -\infty \) while \( \liminf m/n > 0 \)
            as \( n \to \infty \),
            \[
                \mathbb P( (n,m)\text{-digraph is elementary})
                \sim
                e^{-\mu^3/6}
                \sqrt{\dfrac{3\pi}{2}}
                \sum_{q \geq 0} 3^{-q}\ \widehat e_q^{\ (-1)} s_q^+(\mu),
            \]
            where the coefficients \( \widehat e_q^{\ (-1)} \) are given by
            \[
                \sum_{q \geq 0} \widehat e_q^{\ (-1)} y^q
                :=
                \dfrac{1}{\frac{y}{2} + E(y) + 3 y^2 E'(y)}.
            \]
            In particular, when \( \mu \to -\infty \), \( |\mu| \ll n^{-1/3}
            \),
            \[
                \mathbb P( (n,m)\text{-digraph is elementary})
                \sim
                1 - \dfrac{1}{2 |\mu|^3}.
            \]
        \end{theorem}
        \begin{proof}
            The key ingredient is the exact expression
            from~\cref{lemma:number:subcritical:dags}.
            As in the proof of~\cref{theorem:asymptotics:dags}, we can drop the
            terms that give negligible contributions and develop the coefficient
            operator \( [y^t] \) accordingly. The key difference between the
            proofs is the form of the denominator: after taking out a common
            multiple \( (1 - T(z_1)) \) (ignoring higher powers in variable \( y
            \)), the denominator can be again regarded as a formal power
            series in \( \frac{y}{(1 - T(z_1))^{3}} \). In order to obtain the
            asymptotics, the transformed expression should be developed,
            then~\cref{lemma:bdg} (bivariate semi-large powers) is applied, and
            finally the sums are decoupled. For the sum corresponding to
            variable \( z_0 \), we apply again the hypergeometric summation
            formula from~\cite{janson1993birth}.
            In order to settle the subcritical case \( \mu \to -\infty \), we
            apply the asymptotic approximation of \( s_q^+(\mu) \)
            from~\cref{remark:H}.
        \end{proof}

        \begin{remark}
            Curiously enough, the coefficient \( 1/2 \) in the subcritical
            probability can be given the same interpretation as a similar
            coefficient \( 5/24 \) arising in the probability that a random
            graph does not contain a complex component: namely the compensation
            factor of the simplest cubic forbidden multigraph.
        \end{remark}

    \end{subsection}
\end{section}

\paragraph*{Acknowledgements.}
We are grateful to
Olivier Bodini, Naina Ralaivaosaona,
Vonjy Rasendrahasina,
Vlady Ravelomanana, and Stephan Wagner
for fruitful discussions, and to the anonyomous referees whose suggestions
helped to improve this paper.

\printbibliography

\end{document}